\documentclass[12pt,draft]{amsart}
\usepackage{amsmath}
\usepackage{amssymb}
\usepackage{amsthm}
\usepackage{bm}
\usepackage{url}
\numberwithin{equation}{section}
\newtheorem{theorem}{Theorem}[section]

\newtheorem{corollary}[theorem]{Corollary}

\theoremstyle{definition}
\newtheorem{remark}[theorem]{Remark}

\allowdisplaybreaks
\begin{document}
\title[RCSRs of Symmetric Tornheim double zeta functions]{Rapidly convergent series representations of symmetric Tornheim double zeta functions}
\author[T.~Nakamura]{Takashi Nakamura}
\address[T.~Nakamura]{Department of Liberal Arts, Faculty of Science and Technology, Tokyo University of Science, 2641 Yamazaki, Noda-shi, Chiba-ken, 278-8510, Japan}
\email{nakamuratakashi@rs.tus.ac.jp}
\urladdr{https://sites.google.com/site/takashinakamurazeta/}
\begin{abstract}
In the present paper, for $s,t,u \in {\mathbb{C}}$, we show rapidly (or globally) convergent series representations of the Tornheim double zeta function $T(s,t,u)$ and (desingularized) symmetric Tornheim double zeta functions. As a corollary, we give a new a proof of known results on the values of $T(s,s,s)$ at non-positive integers and the location of the poles of $T(s,s,s)$. Furthermore, we prove that the function $T(s,s,s)$ can not be written by a polynomial in the form of $\sum_{k=1}^j c_k \prod_{r=1}^q \zeta^{d_{kr}} (a_{kr} s + b_{kr})$, where $a_{kr}, b_{kr}, c_k \in {\mathbb{C}}$ and $d_{kr} \in {\mathbb{Z}}_{\ge 0}$. 
\end{abstract}

\maketitle

\section{Introduction and main results}
\subsection{Introduction}
For $s,t,u \in {\mathbb{C}}$, we define the Tornheim double zeta function by
\begin{equation}\label{eq:defA_2}
T(s,t,u) :=
\sum_{m,n=1}^\infty \frac{1}{m^s n^t (m+n)^u}
\end{equation}
in the region of absolute convergence $\Re (s+u) >1$, $\Re (t+u) > 1$ and $\Re (s+t+u) > 2$. Clearly, this function is a double sum and triple variables version of the Riemann zeta function $\zeta (s)$. In \cite[Theorem 1]{MaM02} and \cite[Theorem 6.1]{MNOT}, it is showed that $T(s,t,u)$ can be continued meromorphically and its true (not possible) singularities\footnote{For instance, on $s+t=1$, the function $\zeta (s+t) - (s+t-1)^{-1}$ has a possible (not true) singularity but the function $2\zeta(s+t) - (s+t-1)^{-1}$ has a true (not possible) singularity. In general, it is difficult to determine whether singularities of multivariable (zeta) functions are true or not.} are only on the hyperplanes given by one of the following equations:
\begin{equation}\label{eq:TSing1}
s+u \in {\mathbb{Z}}_{\le 1}, \qquad t+u \in {\mathbb{Z}}_{\le 1}, \qquad s+t+u=2.
\end{equation}

The values of the Tornheim double zeta function $T(a,b,c)$ for $a,b,c \in{\mathbb{N}}$ were investigated by Tornheim in 1950, later by Mordell in 1958, and many researchers (see e.g.~\cite[Section 1]{Na}). As a rapidly convergent series representation\footnote{In \cite[Section 7.8]{Wiki}, we can find a dozen rapidly (or globally) convergent series representations of the Riemann zeta function $\zeta (s)$. } of $T(s,t,u)$, Crandall gave the following formula in \cite[Proposition 1]{Bo} (see also \cite[Theorems 2.5 and 2.6]{BoDi}). Let $0 < \theta < 2\pi$ be with a free parameter, $s,t \in {\mathbb{C}} \setminus {\mathbb{Z}}$, $u \in {\mathbb{C}}$ and $\Gamma(u,\theta) := \int_\theta^\infty \omega^{u-1} e^{-\omega} d\omega$ be the upper incomplete Gamma function. Then one has
\begin{equation}\label{eq:Bo1}
\begin{split}
&\Gamma (u) T(s,t,u) = \sum_{m,n=1}^\infty \frac{\Gamma (u, (m+n)\theta)}{m^s n^t (m+n)^u} +
\sum_{k,l=0}^\infty (-1)^{k+l} \frac{\zeta (s-k) \zeta (t-l) \theta^{k+l+1}}{k! l! (k+l+u)} \\ &
+ \Gamma (1-s) \sum_{r=0}^\infty (-1)^r \frac{\zeta (t-r) \theta^{s+u+r-1}}{r! (s+u+r-1)}
+ \Gamma (1-t) \sum_{r=0}^\infty (-1)^r \frac{\zeta (s-r) \theta^{t+u+r-1}}{r! (t+u+r-1)} \\ &
+ \Gamma (1-s) \Gamma (1-t) \frac{\theta^{s+t+u-2}}{s+t+u-2} .
\end{split}
\end{equation}

Define symmetric Tornheim double zeta functions by
\begin{equation*}
\begin{split}
&S_1 (s,t,u) := \bigl( 1 + e^{-\pi i(s+t+u)} \bigr) T(s,t,u)  \\
&+\bigl( e^{-\pi is} + e^{-\pi i(u+t)} \bigr) T(u,s,t) + 
\bigl( e^{-\pi it} + e^{-\pi i(u+s)} \bigr)T(t,u,s),
\end{split}
\end{equation*}
\begin{equation*}
\begin{split}
&S_2 (s,t,u) := \bigl( e^{-\pi iu} + e^{-\pi i(s+t)} \bigr) T(s,t,u) \\
&+ \bigl( e^{-\pi it} + e^{-\pi i(u+s)} \bigr) T(u,s,t) + 
\bigl( e^{-\pi is} + e^{-\pi i(t+u)} \bigr)T(t,u,s) ,
\end{split}
\end{equation*}
for $s,t,u \in {\mathbb{C}}$. Rapidly convergent series representations of $S_1(s,t,u)$ and $S_2(s,t,u)$ are given in \cite[Theorems 1.1 and 1.2]{NH}. To state it, put $\beta (s,t) := 1/B(s,t)$, where $B(s,t)$ is the beta function and $g(s,t,u) := e^{-\pi i(s+t+u)/2}(2\pi)^{s+t+u}$, and for $k \in {\mathbb{Z}}_{\ge 0}$, define the functions $G(s,t,u)$ and $Z_k(s)$ by
\begin{equation}\label{eq:G}
G(s,t,u) := \frac{g(s,t,u)}{\Gamma (s) \Gamma (t) \Gamma (u)} ,\qquad
Z_k(s) := \binom{k-s}{k} \bigl( \zeta (k+1-s) - 1 \bigr).
\end{equation}
Note that $Z_k(s) / \Gamma (s)$ is analytic for all $k \in {\mathbb{Z}}_{\ge 0}$ and $s \in {\mathbb{C}}$. 
Then, except for the true singularities $s+t \in {\mathbb{Z}}_{\le 1}$, we have
\begin{equation}\label{eq:mth1}
\begin{split}
&S_1 (s,t,u) =  G(s,t,u) \Biggl\{ \frac{1}{\beta( u, s \!+\! t \!-\! 1)} + 
\sum_{l,m,n=0}^\infty \frac{Z_l(s) Z_m(t) Z_n(u)}{\beta (n \!+\!1, l \!+\! m \!+\! 1)}  \\
&+ \sum_{l,m=0}^\infty \frac{Z_l(s) Z_m (t)}{u+l+m} 
+ \sum_{l,n=0}^\infty \frac{Z_l(s) Z_n(u)}{\beta( l \!+\! 1, t \!+\! n )}
+ \sum_{m,n=0}^\infty \frac{Z_m(t) Z_n(u)}{\beta( m \!+\! 1, s \!+\! n )} \\ 
&+ \sum_{n=0}^\infty \frac{Z_n(u)}{s \!+\! t \!+\! n \!-\! 1} + \sum_{l=0}^\infty \frac{Z_l(s)}{\beta( t, u \!+\! l)} + 
\sum_{m=0}^\infty \frac{Z_m(t)}{\beta( s, u \!+\! m)} \Biggr\} .
\end{split}
\end{equation}
Furthermore, except for the true singularities $s+t \in {\mathbb{Z}}_{\le 1}$, $t+u \in {\mathbb{Z}}_{\le 1}$, $u+s \in {\mathbb{Z}}_{\le 1}$ or $s+t+u=2$, it holds that
\begin{equation}\label{eq:mth2}
\begin{split}
&\qquad S_2 (s,t,u) = G(s,t,u) \Biggl\{ \frac{1}{s+t+u-2}+ \sum_{l,m,n=0}^\infty \frac{Z_l(s) Z_m(t) Z_n(u)}{l+m+n+1}\\
&+ \sum_{l,m=0}^\infty \frac{Z_l(s) Z_m(t)}{\beta( u, l \!+\! m \!+\! 1)} 
+ \sum_{m,n=0}^\infty \frac{Z_m (t) Z_n (u)}{\beta( s, m \!+\! n \!+\! 1)} 
+ \sum_{n,l=0}^\infty \frac{Z_n (u) Z_l (s)}{\beta( t, n \!+\! l \!+\! 1)} \\
&+\sum_{l=0}^\infty \frac{Z_l(s)}{\beta (l \!+\!1, t \!+\! u \!-\! 1)} 
+ \sum_{m=0}^\infty \frac{Z_m(t)}{\beta (m \!+\!1, u \!+\! s \!-\! 1)} 
+ \sum_{n=0}^\infty \frac{Z_n (u)}{\beta (n \!+\!1, s \!+\! t \!-\! 1)} \Biggr\}  .
\end{split}
\end{equation}
When $k \in {\mathbb{N}}$ is sufficiently large, one has $|\zeta (k+1-s)-1| \le 2^{\Re s -k}$ from the series expression of $\zeta (s)$. Therefore, all the series in (\ref{eq:mth1}) and (\ref{eq:mth2}) converge rapidly. 

\subsection{Main results}
As improvements of (\ref{eq:mth1}) and (\ref{eq:mth2}), we first give the following rapidly convergent series representations of $S_1 (s,t,u)$ and $S_2 (s,t,u)$. It should be noted that the formulas (\ref{eq:Bo1}) and (\ref{eq:mth1}) or (\ref{eq:mth2}) contain the upper incomplete Gamma function and the beta function respectively, but (\ref{eq:nmth1}) and (\ref{eq:nmth2}) contain neither. Recall that $G(s,t,u)$ is defined in (\ref{eq:G}). And put
\[
\eta_k^\pm (s) := (\pm 2)^{-k} \binom{k-s}{k}  \zeta (k+1-s), \qquad \kappa_{l,m,n} = 
\begin{cases}
1 & l+m+n \mbox{ is even,} \\
0 & \mbox{otherwise}. 
\end{cases}
\]
It should be noted that $\eta_k^\pm (s) / \Gamma (s)$ is also analytic for all $k \in {\mathbb{Z}}_{\ge 0}$ and $s \in {\mathbb{C}}$. Then, we have the following  (see also Remark \ref{Rem:16}). 
\begin{theorem}\label{th:nm1}
Except for the true singularities $s+t \in {\mathbb{Z}}_{\le 1}$, we have
\begin{equation}\label{eq:nmth1}
\begin{split}
& S_1 (s,t,u) = G(s,t,u) \Biggl\{ \sum_{l,m,n=0}^\infty \frac{\kappa_{l,m,n} \eta_l^+(s) \eta_m^+(t) \eta_n^-(u)}{l+m+n+1} 
+ \sum_{n=0}^\infty \frac{2^{1-s-t} \eta_n^+(u)}{s+t+n-1} \\ & \qquad +
\sum_{m,n=0}^\infty \! \frac{\eta_m^-(t) \eta_n^+ (u)}{2^s(s+m+n)} 
+ \sum_{l,n=0}^\infty \! \frac{\eta_l^- (s) \eta_n^+ (u)}{2^t (t+l+n)} 
+ \sum_{l,m=0}^\infty \! \frac{\eta_l^+(s) \eta_m^+(t)}{2^u(u+l+m)} \Biggr\} .
\end{split}
\end{equation}

Furthermore, except for the true singularities $s+t \in {\mathbb{Z}}_{\le 1}$, $t+u \in {\mathbb{Z}}_{\le 1}$, $u+s \in {\mathbb{Z}}_{\le 1}$ or $s+t+u=2$, it holds that
\begin{equation}\label{eq:nmth2}
\begin{split}
&S_2 (s,t,u) = G(s,t,u) \Biggl\{  
\frac{2^{2-s-t-u}}{s+t+u-2} + \sum_{l,m,n=0}^\infty \frac{\kappa_{l,m,n} \eta_l^+(s) \eta_m^+(t) \eta_n^+(u)}{l+m+n+1} \\
&+ \sum_{l,m=0}^\infty \! \frac{\eta_l^-(s) \eta_m^- (t)}{2^u(u+l+m)} +
\sum_{m,n=0}^\infty \! \frac{\eta_m^-(t) \eta_n^- (u)}{2^s(s+m+n)}
+ \sum_{n,l=0}^\infty \! \frac{\eta_n^- (u) \eta_l^- (s) }{2^t (t+n+l)} \\
&+\sum_{l=0}^\infty \frac{2^{1-t-u} \eta_l^-(s)}{t+u+l-1} + \sum_{m=0}^\infty \frac{2^{1-u-s} \eta_m^-(t)}{u+s+m-1} 
+\sum_{n=0}^\infty \frac{2^{1-s-t} \eta_n^-(u)}{s+t+n-1} \Biggr\}  .
\end{split}
\end{equation}
\end{theorem}

Next, we define the following symmetric Tornheim double zeta functions $S_3 (s,t,u)$ and $S_4 (s,t,u)$ by
\begin{equation*}
\begin{split}
&S_3 (s,t,u) :=  T(s,t,u) + T(u,s,t) + T(t,u,s),\\
&S_4 (s,t,u) := -T(s,t,u) + T(u,s,t) + T(t,u,s),
\end{split}
\end{equation*}
which are treated in \cite[Proposition 2.1]{EM}. Let ${\mathbb{Z}}_{\le 1}^{od}$ be the set of all odd integers smaller than or equal to $1$ and ${\mathbb{Z}}_{\le 0}^{ev}$ be the set of all non-positive even integers. And put
\begin{equation*}
\begin{split}
&\Gamma_{\!\! \rm{cos}} (s) : = \frac{2\Gamma (s)}{(2\pi)^s}  \cos \Bigl( \frac{\pi s}{2} \Bigr),  \qquad \qquad \quad \,\,\,\,
\Gamma_{\!\! \rm{sin}} (s) : = \frac{2\Gamma (s)}{(2\pi)^s}  \sin \Bigl( \frac{\pi s}{2} \Bigr),\\
&G_{\! ccc} (s,t,u) := \frac{4}{\Gamma_{\!\! \rm{cos}} (s) \Gamma_{\!\! \rm{cos}} (t) \Gamma_{\!\! \rm{cos}} (u)}, \qquad
G_{\! ssc} (s,t,u) := \frac{4}{\Gamma_{\!\! \rm{sin}} (s) \Gamma_{\!\! \rm{sin}} (t) \Gamma_{\!\! \rm{cos}} (u)}.
\end{split}
\end{equation*}
Then, we have the following  (see also Remark \ref{Rem:17}).

\begin{theorem}\label{th:nm2}
Except for the true singularities $s+t \in {\mathbb{Z}}_{\le 1}^{od}$, $t+u \in {\mathbb{Z}}_{\le 1}^{od}$, $u+s \in {\mathbb{Z}}_{\le 1}^{od}$ or $s+t+u=2$, one has
\begin{equation}\label{eq:nmth3}
\begin{split}
& S_3(s,t,u) = G_{\! ccc} (s,t,u) \Biggl\{  
\frac{2^{-s-t-u}}{s+t+u-2} + \sum_{l,m,n=0}^\infty \frac{\eta_{2l}^+(s) \eta_{2m}^+(t) \eta_{2n}^+(u)}{2l+2m+2n+1} \\
&+ \sum_{l,m=0}^\infty \! \frac{2^{-u} \eta_{2l}^+(s) \eta_{2m}^+ (t)}{u+2l+2m} +
\sum_{m,n=0}^\infty \! \frac{2^{-s} \eta_{2m}^+(t) \eta_{2n}^+(u)}{s+2m+2n}
+ \sum_{n,l=0}^\infty \! \frac{2^{-t} \eta_{2n}^+(u) \eta_{2l}^+(s)}{t+2n+2l} \\
&+\sum_{l=0}^\infty \frac{2^{-t-u} \eta_{2l}^+(s)}{t+u+2l-1} + \sum_{m=0}^\infty \frac{2^{-u-s} \eta_{2m}^+(t)}{u+s+2m-1} 
+\sum_{n=0}^\infty \frac{2^{-s-t} \eta_{2n}^+(u)}{s+t+2n-1} \Biggr\} .
\end{split}
\end{equation}

Furthermore, except for the true singularities $s+t \in {\mathbb{Z}}_{\le 1}^{od}$, $t+u \in {\mathbb{Z}}_{\le 0}^{ev}$, $u+s \in {\mathbb{Z}}_{\le 0}^{ev}$ or $s+t+u=2$, it holds that
\begin{equation}\label{eq:nmth4}
\begin{split}
&S_4 (s,t,u) = G_{\! ssc} (s,t,u) \Biggl\{  
\frac{2^{-s-t-u}}{s+t+u-2} + \sum_{l,m,n=0}^\infty \frac{\eta_{2l+1}^+(s) \eta_{2m+1}^+(t) \eta_{2n}^+(u)}{2l+2m+2n+3} \\
&+ \sum_{l,m=0}^\infty \! \frac{2^{-u} \eta_{2l+1}^+(s) \eta_{2m+1}^+ (t)}{u+2l+2m+2} -
\sum_{m,n=0}^\infty \! \frac{2^{-s} \eta_{2m+1}^+(t) \eta_{2n}^+(u)}{s+2m+2n+1}
- \sum_{n,l=0}^\infty \! \frac{2^{-t} \eta_{2n}^+(u) \eta_{2l+1}^+(s)}{t+2n+2l+1} \\
&-\sum_{l=0}^\infty \frac{2^{-t-u} \eta_{2l+1}^+(s)}{t+u+2l} - \sum_{m=0}^\infty \frac{2^{-u-s} \eta_{2m+1}^+(t)}{u+s+2m} 
+\sum_{n=0}^\infty \frac{2^{-s-t} \eta_{2n}^+(u)}{s+t+2n-1} \Biggr\} .
\end{split}
\end{equation}
\end{theorem}

By using the theorems above, we show the following four corollaries. 
The eight equations below are analogues of (\ref{eq:Bo1}) and \cite[Corollary 1.1]{NH}, which give rapidly convergent series representations or explicit evaluation formulas for $T(s,t,u)$.
\begin{corollary}\label{cor:nm1}
Except for the true singularities $t+u \in {\mathbb{Z}}_{\le 1}$ or $u+s \in {\mathbb{Z}}_{\le 1}$, one has
\begin{equation} \tag{i}
\begin{split}
&(1 - e^{-2\pi is}) (1 - e^{-2\pi it}) (e^{-\pi i(s+t+u)}-1) T(s,t,u) = \\
& (e^{-2\pi i(s+t)}-1) S_1(s,t,u) + e^{-\pi is} (1-e^{-2\pi it}) S_1(u,s,t) + e^{-\pi it} (1-e^{-2\pi is}) S_1(t,u,s) . 
\end{split}
\end{equation}
Except for the true singularities $t+u \in {\mathbb{Z}}_{\le 1}$, $u+s \in {\mathbb{Z}}_{\le 1}$ or $s+t+u=2$, 
\begin{equation} \tag{ii}
\begin{split}
&(1 - e^{-2\pi is}) (1 - e^{-2\pi it}) (e^{-\pi i(s+t)} - e^{-\pi iu}) T(s,t,u) =\\
&(e^{-2\pi i(s+t)}-1) S_2(s,t,u) + e^{-\pi it} (1-e^{-2\pi is}) S_1(u,s,t) + e^{-\pi is} (1-e^{-2\pi it}) S_1(t,u,s) ,
\end{split}
\end{equation}
\begin{equation} \tag{iii}
2T(s,t,u) = S_3(s,t,u) - S_4(s,t,u),
\end{equation}
\begin{equation} \tag{iv}
\begin{split}
&(1 - e^{-\pi is}) (1 - e^{-\pi it}) (1+e^{-\pi iu}) T(s,t,u) = \\
& (1+e^{-\pi iu})(1+e^{-\pi i(s+t)}) S_3 (s,t,u) - S_1(u,s,t) - S_1(t,u,s),
\end{split}
\end{equation}
\begin{equation} \tag{v}
\begin{split}
&(1 + e^{-\pi is}) (1 + e^{-\pi it}) (1+e^{-\pi iu}) T(s,t,u) = \\
& -(1+e^{-\pi iu})(1 + e^{-\pi i(s+t)}) S_3 (s,t,u) + S_1(u,s,t) + S_1(t,u,s),
\end{split}
\end{equation}
\begin{equation} \tag{vi}
\begin{split}
&(1 - e^{-\pi is}) (1 - e^{-\pi it}) (1 + e^{-\pi iu}) T(s,t,u) = \\
& S_1 (s,t,u) + S_2(s,t,u) - (e^{-\pi is} + e^{-\pi it})(1 + e^{-\pi iu)}) S_3(s,t,u),
\end{split}
\end{equation}
\begin{equation} \tag{vii}
\begin{split}
&(1 + e^{-\pi is}) (1 + e^{-\pi it}) (1 + e^{-\pi iu}) T(s,t,u) = \\
& S_1 (s,t,u) + S_2(s,t,u) - (e^{-\pi is} + e^{-\pi it})(1+e^{-\pi iu}) S_4(s,t,u).
\end{split}
\end{equation}
\begin{equation} \tag{viii}
2T(s,t,u) = S_4(u,s,t) + S_4(t,u,s).
\end{equation}
\end{corollary}

Moreover, we give the values of $T(s,s,s)$ at non-positive integers, which is a special case of \cite[Theorem 1]{OR} (see also \cite[Theorem 1.3]{Romik}). 
\begin{corollary}\label{cor:nm2}
Let $k$ be a negative integer. Then we have
\[
\lim_{\varepsilon \to 0} T(\varepsilon, \varepsilon, \varepsilon) = \frac{1}{3}, \qquad 
\lim_{\varepsilon \to 0} T(k+\varepsilon, k+\varepsilon, k+\varepsilon) = 0.
\]
\end{corollary}

We give a new proof of the following statement proved in \cite[Theorem 1.2 (ii)]{Romik}.
\begin{corollary}\label{cor:pole1}
All the poles of $T(s,s,s)$ are simple and only at
\begin{equation}\label{eq:pole1}
s=2/3 \qquad \mbox{and} \qquad s=1/2-k, \quad k \in {\mathbb{Z}}_{\ge 0}.
\end{equation}
\end{corollary}

Furthermore, we can show the following (see also Remark \ref{Rem:18}).
\begin{corollary}\label{cor:nm3}
The function $T(s,s,s)$ can not be written by a polynomial in the form of
\begin{equation}\label{eq:poly1}
\sum_{k=1}^j c_k \prod_{r=1}^q \zeta^{d_{kr}} (a_{kr} s + b_{kr}) ,\qquad 
a_{kr}, b_{kr}, c_k \in {\mathbb{C}}, \quad d_{kr} \in {\mathbb{Z}}_{\ge 0} .
\end{equation}
\end{corollary}

\begin{remark}\label{Rem:16}
When $-s>0$ and $-u>0$ are large, there is a possibility that the rapidly convergent series representation deduced by (\ref{eq:mth1}) have some advantage over the Crandall expansion (\ref{eq:Bo1}) in view of numerical calculation (see \cite[Remark 1.1]{NH}). In addition, if $\sigma >1$, one has
\[
\zeta (\sigma) - 1 = \sum_{k=2}^\infty \frac{1}{k^\sigma} > \sum_{k=1}^\infty \frac{1}{(2k)^\sigma} = \frac{\zeta(\sigma)}{2^\sigma} .
\]
Hence, for $k+1-\sigma >1$, it holds that
\[
0 < \bigl | 2^{\sigma-1} \eta_k^\pm (\sigma) \bigr| < \bigl| Z_k (\sigma) \bigr| .
\]
Therefore, some series in (\ref{eq:nmth1}) and (\ref{eq:nmth2}) converge faster than those in (\ref{eq:mth1}) and (\ref{eq:mth2}). Moreover, the rapidly convergent series representation (\ref{eq:mth1}) has seven infinite series but (\ref{eq:nmth1}) has only five infinite series. 
\end{remark}

\begin{remark}\label{Rem:17}
It should be emphasised that from Theorems \ref{th:nm1} and \ref{th:nm2}, 

TS (true singularities) of $S_1 (s,t,u)$ are only on $s+t \in {\mathbb{Z}}_{\le 1}$, 

TS of $S_3 (s,t,u)$ are only on $s+t \in {\mathbb{Z}}_{\le 1}^{od}$, $t+u \in {\mathbb{Z}}_{\le 1}^{od}$, $u+s \in {\mathbb{Z}}_{\le 1}^{od}$ and $s+t+u=2$, 

TS of $S_4 (s,t,u)$ are only on $s+t \in {\mathbb{Z}}_{\le 1}^{od}$, $t+u \in {\mathbb{Z}}_{\le 0}^{ev}$, $u+s \in {\mathbb{Z}}_{\le 0}^{ev}$ and $s+t+u=2$ \\
even though

TS of $T(s,t,u)$ are on $t+u \in {\mathbb{Z}}_{\le 1}$, $u+s \in {\mathbb{Z}}_{\le 1}$ and $s+t+u=2$,

TS of $S_2(s,t,u)$ are on $s+t \in {\mathbb{Z}}_{\le 1}$, $t+u \in {\mathbb{Z}}_{\le 1}$, $u+s \in {\mathbb{Z}}_{\le 1}$ and $s+t+u=2$.\\
Therefore, we can regard $S_1 (s,t,u)$, $S_3 (s,t,u)$ and $S_4 (s,t,u)$ as desingulazations of the Tornheim double zeta function, which are analogues of desingularized multiple zeta functions given by Furusho Komori Matsumoto and Tsumura (see \cite{FKMT1} and \cite{FKMT2}). 
\end{remark}

\begin{remark}\label{Rem:18}
Corollary \ref{cor:nm3} should be compared with the equations that 
\[
T(0,0,s) = \zeta (s-1) - \zeta (s), \qquad 2T(0,s,s) = \zeta^2(s)- \zeta (2s)
\]
(see \cite[(15) and (16) in p.~219]{SriCho}), which means that the Barnes double zeta function $T(0,0,s)$ and the Euler-Zagier double zeta function with identical arguments $T(0,s,s)$ can be written as polynomials in the form of (\ref{eq:poly1}). It also should be mentioned that, for $a,b \in {\mathbb{N}}$ and $s \in {\mathbb{C}}$, the function $T(a,b,s)+(-1)^b T(b,s,a)+(-1)^a T(s,a,b)$ can be expressed in the form of $(a! b!)^{-1}\sum_{k=0}^j c_k \zeta (2k) \zeta (a + b + s -2k)$, where $c_k \in {\mathbb{N}}$ (see e.g.~\cite[Theorem 1.2]{Na}). 
\end{remark}

\section{Proofs}
\subsection{Preliminaries}
Recall the definition and some basic properties for the Hurwitz zeta function and the periodic zeta functions. For $a>0$ and $\Re (s) > 1$, the Hurwitz zeta function $\zeta (s,a)$ and the periodic zeta function $F(s,a)$ are defined by 
\[
\zeta (s,a) := \sum_{n=0}^\infty \frac{1}{(n+a)^s}, \qquad F(s,a) := \sum_{n=1}^\infty \frac{e^{2\pi ina}}{n^s},
\]
respectively. The Hurwitz zeta function $\zeta (s,a)$ is a meromorphic function with a simple pole at $s=1$ whose residue is $1$ (e.g.~\cite[Section 12]{Apo}). Note that the function $F(s,a)$ with $a \not \in {\mathbb{N}}$ is analytically continuable to the whole complex plan since the Dirichlet series of $F(s,a)$ converges uniformly in each compact subset of the half-plane $\Re (s) >0$ when $a \not \in {\mathbb{N}}$ (e.g.~\cite[p.~20]{LauGa}). 

The following functional equation is widely known (see \cite[Theorem 12.6]{Apo}). Note that the equations below hold for all $s \in {\mathbb{C}}$ if $0<a<1$. 
\begin{equation}\label{eq:fe1}
e^{-\pi is/2} F(s,a) + e^{\pi is/2} F(s,1-a) = \frac{(2\pi)^s}{\Gamma (s)} \zeta (1-s,a).
\end{equation}
For simplicity, put
\begin{equation*}
\begin{split}
&Z(s,a) := \zeta (s,a) + \zeta (s,1-a), \qquad Y(s,a) := \zeta (s,a) - \zeta (s,1-a),\\
&\quad \Gamma_{\!\! \rm{cos}} (s) : = \frac{2\Gamma (s)}{(2\pi)^s}  \cos \Bigl( \frac{\pi s}{2} \Bigr),  \qquad \qquad
\Gamma_{\!\! \rm{sin}} (s) : = \frac{2\Gamma (s)}{(2\pi)^s}  \sin \Bigl( \frac{\pi s}{2} \Bigr).
\end{split}
\end{equation*}
Then, by \cite[(2.2) and (2.3)]{EM}, we have the functional equations
\begin{equation}\label{eq:feCS}
\sum_{n=1}^\infty \frac{\cos 2 \pi na}{n^s} = \frac{Z (1-s,a)}{2 \Gamma_{\!\! \rm{cos}} (s)}, \qquad
\sum_{n=1}^\infty \frac{\sin 2 \pi na}{n^s} = \frac{Y (1-s,a)}{2 \Gamma_{\!\! \rm{sin}} (s)}
\end{equation}
for $\Re (s) >1$. The following equation is known as a Taylor series expansion of Hurwitz zeta function $\zeta (s,1-a)$ (see \cite[p.~250, (19)]{SriCho}) 
\begin{equation}\label{eq:HuTay1}
G (s,a) := \sum_{k=0}^\infty \binom{k-s}{k} \zeta (1-s+k) a^k = \zeta (1-s,1-a), \qquad |a| <1.
\end{equation}
By this formula and $\zeta (s,a) = a^{-s} + \zeta (s,1+a)$, we can easily show that
\begin{equation}\label{eq:HuTay2} 
\zeta (1-s,a) = a^{s-1} + G(s,-a), \qquad |a| <1.
\end{equation}

\subsection{Proof of Theorem \ref{th:nm1}}
The argument below is inspired by the proof of \cite[Theorem 1.1]{NH} and Onodera's proof of \cite[Theorem 2]{Onodera}. Define the function
\begin{equation}\label{eq:defGH}
H(s,a) :=\frac{e^{-\pi is/2}(2\pi)^s}{\Gamma (s)} \zeta (1-s,a) = e^{-\pi is}F(s,1-a) + F(s,a).
\end{equation}
Moreover, for $\Re(s), \Re (t), \Re (u) >1$, we can easily see that
\begin{equation}\label{eq:Tzagier}
T(s,t,u) = \int_0^1 \sum_{l=1}^\infty \frac{e^{2\pi i la}}{l^s} \sum_{m=1}^\infty \frac{e^{2\pi i ma}}{m^t} 
\sum_{n=1}^\infty \frac{e^{2\pi i n(1-a)}}{n^u}da . 
\end{equation}

\begin{proof}[Proof of (\ref{eq:nmth1})]
By (\ref{eq:fe1}), (\ref{eq:defGH}) and (\ref{eq:Tzagier}), it holds that
\begin{equation*}
\begin{split}
&\bigl( 1 + e^{-\pi i(s+t+u)} \bigr) T(s,t,u) + \bigl( e^{-\pi is} + e^{-\pi i(u+t)} \bigr) T(u,s,t) + 
\bigl( e^{-\pi it} + e^{-\pi i(u+s)} \bigr)T(t,u,s) \\ &= 
\int_0^1 \sum_{l,m,n \in {\mathbb{Z}}_{\ne 0}} \frac{e^{2\pi i la} e^{2\pi i ma} e^{2\pi i n(1-a)}}{l^s m^t n^u} da 
= \int_0^1 H(s,a) H (t,a) H (u,1-a) da \\
&= \int_0^{1/2} H(s,a) H (t,a) H (u,1-a) da + \int_{1/2}^1 H(s,a) H (t,a) H (u,1-a) da 
\end{split}
\end{equation*}
when $\Re(s), \Re (t), \Re (u) >1$. Denote the first and second integral in the right-hand side of the formula above by $I_1$ and $I_2$, respectively. Then, by (\ref{eq:HuTay1}) and (\ref{eq:HuTay2}), we have
\begin{equation*}
\begin{split}
I_1 = G(s,t,u) \int_0^{1/2} \bigl( a^{s-1} + G(s,-a) \bigr) \bigl( a^{t-1} + G(t,-a) \bigr) G(u,a) da .
\end{split}
\end{equation*}
From the definition of $G(s,a)$, it holds that
\[
\int_0^{1/2}  a^{s+t-2} G(u,a) da = \sum_{n=0}^\infty \binom{n-u}{n} \frac{2^{1-s-t-n} \zeta (1-u+n)}{s+t+n-1}
\]
which provides the second infinite series in (\ref{eq:nmth1}). The true singularities $s+t \in {\mathbb{Z}}_{\le 1}$ of $S_1(s,t,u)$ are deduced from the infinite series above. In contrast, we have 
\begin{equation*}
\begin{split}
\int_0^{1/2} a^{s-1} G(t,-a) G(u,a) da = 
\sum_{m,n=0}^\infty \binom{m-t}{m} \binom{n-u}{n} \frac{\zeta (1-t+m) \zeta (1-u+n)}{(-2)^{m} 2^{s+n}(s+m+n)}
\end{split}
\end{equation*}
which yields the third infinite series in (\ref{eq:nmth1}). Note that the poles caused by $1/(s+m+n)$ are canceled by the zeros of the function $1/\Gamma(s)$ in $G(s,t,u)$. By changing variables, we obtain the fourth infinite series in (\ref{eq:nmth1}). Furthermore, it holds that
\begin{equation*}
\begin{split}
&\int_0^{1/2} G(s,-a) G(t,-a) G(u,a) da = \\ & \sum_{l,m,n=0}^\infty 
\binom{l-s}{l} \binom{m-t}{m} \binom{n-u}{n} \frac{\zeta (1-s+l) \zeta (1-t+m) \zeta (1-u+n)}{(-2)^{l+m} 2^{n+1} (l+m+n+1)}.
\end{split}
\end{equation*}
Next, consider the integral $I_2$. Replacing variable $a$ by $1-a$, we have
\begin{equation*}
\begin{split}
I_2 &= \int_0^{1/2} H(s,1-a) H (t,1-a) H (u,a) da \\ &= G(s,t,u) \int_0^{1/2} G(s,a) G(t,a) \bigl( a^{u-1} + G(u,-a) \bigr) da .
\end{split}
\end{equation*}
We obtain the fifth infinite series in (\ref{eq:nmth1}) from
\[
\int_0^{1/2} G(s,a) G(t,a) a^{u-1} da = 
\sum_{l,m=0}^\infty \binom{l-s}{l} \binom{m-t}{m} \frac{\zeta (1-s+l) \zeta (1-t+m)}{2^{u+l+m} (u+l+m)}. 
\]
The poles caused by $1/(u+l+m)$ are canceled by the zeros of $1/\Gamma(u)$. Moreover, one has
\begin{equation*}
\begin{split}
&\int_0^{1/2} G(s,a) G(t,a) G(u,-a) da = \\ & \sum_{l,m,n=0}^\infty 
\binom{l-s}{l} \binom{m-t}{m} \binom{n-u}{n} \frac{\zeta (1-s+l) \zeta (1-t+m) \zeta (1-u+n)}{(-2)^n 2^{l+m+1} (l+m+n+1)}.
\end{split}
\end{equation*}
Thus, we have the first infinite series (\ref{eq:nmth1}) by the integrals $\int_0^{1/2} G(s,-a) G(t,-a) G(u,a) da$ and $\int_0^{1/2} G(s,a) G(t,a) G(u,-a) da$. 
\end{proof}


\begin{proof}[Proof of (\ref{eq:nmth2})]
From (\ref{eq:fe1}), (\ref{eq:defGH}) and (\ref{eq:Tzagier}), it holds that
\begin{equation*}
\begin{split}
&\bigl( e^{-\pi iu} + e^{-\pi i(s+t)} \bigr) T(s,t,u) + \bigl( e^{-\pi it} + e^{-\pi i(u+s)} \bigr) T(u,s,t) + 
\bigl( e^{-\pi is} + e^{-\pi i(t+u)} \bigr)T(t,u,s) \\ 
&= \int_0^1 \sum_{l,m,n \in {\mathbb{Z}}_{\ne 0}} \frac{e^{2\pi i la} e^{2\pi i ma} e^{2\pi i na}}{l^s m^t n^u} da 
= \int_0^1 H(s,a) H (t,a) H (u,a) da \\
&= \int_0^{1/2} H(s,a) H (t,a) H (u,a) da + \int_{1/2}^1 H(s,a) H (t,a) H (u,a) da 
\end{split}
\end{equation*}
for $\Re(s), \Re (t), \Re (u) >1$. Denote the first and second integral in the right-hand side of the formula above by $I_1$ and $I_2$, respectively. Then, by (\ref{eq:HuTay1}) and (\ref{eq:HuTay2}), one has
\begin{equation*}
\begin{split}
I_1 = G(s,t,u) \int_0^{1/2} \bigl( a^{s-1} + G(s,-a) \bigr) \bigl( a^{t-1} + G(t,-a) \bigr) \bigl( a^{u-1} + G(u,-a) \bigr) da .
\end{split}
\end{equation*}
Clearly, we have the first term in (\ref{eq:nmth2}) from
\begin{equation}\label{eq:ft1}
\int_0^{1/2} a^{s+t+u-3} da = \frac{2^{2-s-t-u}}{s+t+u-2} .
\end{equation}
By the definition of $G(s,-a)$, it holds that
\[
\int_0^{1/2} a^{s+t-2} G(u,-a) da = \sum_{n=0}^\infty \binom{n-u}{n} \frac{2^{1-s-t} \zeta (1-u+n)}{(-2)^n (s+t+n-1)}
\]
which coincides with the seventh infinite series in (\ref{eq:nmth2}). The true singularities $s+t \in {\mathbb{Z}}_{\le 1}$ of $S_2(s,t,u)$ are come from the infinite series above. By changing variables, we obtain the fifth and sixth infinite series in (\ref{eq:nmth2}). By contrast, we have 
\begin{equation*}
\begin{split}
\int_0^{1/2} a^{s-1} G(t,-a) G(u,-a) da = 
\sum_{m,n=0}^\infty \binom{m-t}{m} \binom{n-u}{n} \frac{\zeta (1-t+m) \zeta (1-u+n)}{2^s (-2)^{m+m} (s+m+n)}
\end{split}
\end{equation*}
which provides the third infinite series in (\ref{eq:nmth2}). It should be mentioned that the poles caused by $1/(s+m+n)$ are canceled by the zeros of $1/\Gamma(s)$. By changing variables, we obtain the second and fourth infinite series in (\ref{eq:nmth2}). Moreover, we have
\begin{equation*}
\begin{split}
&\int_0^{1/2} G(s,-a) G(t,-a) G(u,-a) da = \\ & \sum_{l,m,n=0}^\infty 
\binom{l-s}{l} \binom{m-t}{m} \binom{n-u}{n} \frac{\zeta (1-s+l) \zeta (1-t+m) \zeta (1-u+n)}{2 (-2)^{m+n+l} (l+m+n+1)}.
\end{split}
\end{equation*}
Next, consider the integral $I_2$. Replacing variable $a$ by $1-a$, we have
\begin{equation*}
\begin{split}
I_2 = \int_0^{1/2} \!\! H(s,1-a) H (t,1-a) H (u,1-a) da = G(s,t,u) \int_0^{1/2} \!\! G(s,a) G(t,a) G(u,a) da .
\end{split}
\end{equation*}
Then, it holds that
\begin{equation*}
\begin{split}
&\int_0^{1/2} G(s,a) G(t,a) G(u,a) da = \\ & \sum_{l,m,n=0}^\infty 
\binom{l-s}{l} \binom{m-t}{m} \binom{n-u}{n} \frac{\zeta (1-s+l) \zeta (1-t+m) \zeta (1-u+n)}{2^{l+m+n+1} (l+m+n+1)}.
\end{split}
\end{equation*}
Hence, we have the first infinite series in (\ref{eq:nmth2}) by $\int_0^{1/2} G(s,-a) G(t,-a) G(u,-a) da + \int_0^{1/2} G(s,a) G(t,a) G(u,a) da $. 
\end{proof}

\subsection{Proof of Theorem \ref{th:nm2}}
The argument below is partially based on the proof of \cite[Proposition 2.1]{EM}.  

\begin{proof}[Proof of (\ref{eq:nmth3})]
It is well-known that
\[
4\cos x \cos y \cos z = \cos(x+y+z) + \cos(-x+y+z) + \cos(x-y+z) + \cos(x+y-z)
\]
for $x,y,z \in {\mathbb{C}}$. Thus, by $Z(s,a)=Z(s,1-a)$ and the first equation in (\ref{eq:feCS}), we have
\begin{equation*}
\begin{split}
& T(s,t,u) + T(u,s,t) + T(t,u,s) = \int_0^1 \sum_{l,m,n >0} \frac{4 \cos 2\pi i la \cos 2\pi i ma \cos 2\pi i na}{l^s m^t n^u} da \\
&= \int_0^1 \frac{Z(1-s,a) Z(1-t,a) Z(1-u,a) da}{2 \Gamma_{\!\! \rm{cos}} (s) \Gamma_{\!\! \rm{cos}} (t) \Gamma_{\!\! \rm{cos}} (u)} 
=  \int_0^{1/2} \frac{Z(1-s,a) Z(1-t,a) Z(1-u,a) da}{\Gamma_{\!\! \rm{cos}} (s) \Gamma_{\!\! \rm{cos}} (t) \Gamma_{\!\! \rm{cos}} (u)} 
\end{split}
\end{equation*}
when $\Re(s), \Re (t), \Re (u) >1$. Obviously, one has
\[
Z(1-s,a) = a^{s-1} + \zeta (1-s,1+a) + \zeta(1-s,1-a) = a^{s-1} + G(s,-a) + G(s,a). 
\]
Then, by (\ref{eq:HuTay1}) and (\ref{eq:HuTay2}), we have
\begin{equation}\label{eq:defZ}
Z(1-s,a) = a^{s-1} + G_1(s,a), \qquad G_1(s,a) := 2 \sum_{k=0}^\infty \binom{2k-s}{2k} \zeta (1-s+2k) a^{2k} .
\end{equation}
Therefore, we consider the integral expressed as
\[
\int_0^{1/2} \bigl( a^{s-1} + G_1(s,a) \bigr) \bigl( a^{t-1} + G_1(t,a) \bigr) \bigl( a^{u-1} + G_1(u,a) \bigr) da. 
\]

We obtain the first term within the parenthesis of the right-hand side of (\ref{eq:nmth3}) from (\ref{eq:ft1}). By the definition of $G_1(s,a)$, we have
\begin{equation}\label{eq:odd1}
\int_0^{1/2} a^{s+t-2} G_1(u,a) da = 2 \sum_{n=0}^\infty \binom{2n-u}{2n} \frac{\zeta (1-u+2n)}{2^{s+t+2n-1} (s+t+2n-1)}
\end{equation}
which yields the seventh infinite series in (\ref{eq:nmth3}). The true singularities $s+t \in {\mathbb{Z}}_{\le 1}^{od}$ of $S_3(s,t,u)$ are come from the infinite series above.  By changing variables, we obtain the fifth and sixth infinite series in (\ref{eq:nmth3}). By contrast, we have 
\begin{equation*}
\begin{split}
\int_0^{1/2} a^{s-1} G_1(t,a) G_1(u,a) da = 
4 \sum_{m,n=0}^\infty \binom{2m-t}{2m} \binom{2n-u}{2n} \frac{\zeta (1-t+2m) \zeta (1-u+2n)}{2^{s+2m+2m} (s+2m+2n)}
\end{split}
\end{equation*}
which gives the third series in (\ref{eq:nmth3}). It should be noted that the poles caused by $1/(s+2m+2n)$ are canceled by the zeros of $1/\Gamma(s)$. By changing variables, we obtain the second and fourth infinite series in (\ref{eq:nmth3}). Moreover, we obtain
\begin{equation*}
\begin{split}
&\int_0^{1/2} G_1(s,a) G_1(t,a) G_1(u,a) da = \\ & 8 \sum_{l,m,n=0}^\infty 
\binom{2l-s}{2l} \binom{2m-t}{2m} \binom{2n-u}{2n} \frac{\zeta (1-s+2l) \zeta (1-t+2m) \zeta (1-u+2n)}{2^{2l+2m+2n+1} (2l+2m+2n+1)}
\end{split}
\end{equation*}
which coincides with the first infinite series in (\ref{eq:nmth3}). Therefore, we have (\ref{eq:nmth3}). 
\end{proof}


\begin{proof}[Proof of (\ref{eq:nmth4})]
It is widely-known that
\[
4\sin x \sin y \cos z = \cos(-x+y+z) + \cos(x-y+z) - \cos(x+y-z) -\cos(x+y+z)
\]
for $x,y,z \in {\mathbb{C}}$. Hence, from $Z(s,a)=Z(s,1-a)$, $Y(s,a)=-Y(s,1-a)$ and the second functional equation in (\ref{eq:feCS}), it holds that
\begin{equation*}
\begin{split}
& -T(s,t,u) + T(u,s,t) + T(t,u,s) = \int_0^1 \sum_{l,m,n >0} \frac{4 \sin 2\pi i la \sin 2\pi i ma \cos 2\pi i na}{l^s m^t n^u} da \\
&= \int_0^1 \frac{Y(1-s,a) Y(1-t,a) Z(1-u,a) da}{2 \Gamma_{\!\! \rm{sin}} (s) \Gamma_{\!\! \rm{sin}} (t) \Gamma_{\!\! \rm{cos}} (u)} 
=  \int_0^{1/2} \frac{Y(1-s,a) Y(1-t,a) Z(1-u,a) da}{\Gamma_{\!\! \rm{sin}} (s) \Gamma_{\!\! \rm{sin}} (t) \Gamma_{\!\! \rm{cos}} (u)} 
\end{split}
\end{equation*}
when $\Re(s), \Re (t), \Re (u) >1$. Clearly, we have
\[
Y(1-s,a) = a^{s-1} + \zeta (1-s,1+a) - \zeta(1-s,1-a) = a^{s-1} + G(s,-a) - G(s,a). 
\]
Hence, it holds that
\[
Y(1-s,a) = a^{s-1} + G_2(s,a), \qquad G_2(s,a) := -2 \sum_{k=0}^\infty \binom{2k+1-s}{2k+1} \zeta (2-s+2k) a^{2k+1}
\]
by (\ref{eq:HuTay1}) and (\ref{eq:HuTay2}). Thus, we consider the integral expressed as
\[
\int_0^{1/2} \bigl( a^{s-1} + G_2(s,a) \bigr) \bigl( a^{t-1} + G_2(t,a) \bigr) \bigl( a^{u-1} + G_1(u,a) \bigr) da. 
\]

We obtain the first term within the brackets of the right-hand side of (\ref{eq:nmth4}) from (\ref{eq:ft1}). By the definition of $G_2(s,a)$, we obtain 
\[
\int_0^{1/2} a^{t+u-2} G_2(s,a) da = -2 \sum_{n=0}^\infty \binom{2n+1-s}{2n+1} \frac{\zeta (2-s+2n)}{2^{t+u+2n} (t+u+2n)}
\]
which gives the fifth series in (\ref{eq:nmth4}). The true singularities $t+u \in {\mathbb{Z}}_{\le 0}^{ev}$ of $S_4(s,t,u)$ are come from the infinite series above. By changing variables, we obtain the sixth infinite series in (\ref{eq:nmth4}). We obtain the seventh series from (\ref{eq:odd1}). By contrast, we have 
\begin{equation*}
\begin{split}
&\int_0^{1/2} a^{u-1} G_2(s,a) G_2(t,a) da = \\
&4 \sum_{l,m=0}^\infty \binom{2l+1-u}{2l+1} \binom{2m+1-t}{2m+1} \frac{\zeta (2-s+2l) \zeta (2-t+2m)}{2^{u+2l+2m+2} (s+2l+2m+2)}
\end{split}
\end{equation*}
which yields the second infinite series in (\ref{eq:nmth4}). It should be noted that the poles caused by $1/(s+2m+2n+2)$ are canceled by the zeros of $1/\Gamma(s)$. Moreover, it holds that
\begin{equation*}
\begin{split}
&\int_0^{1/2} a^{s-1} G_2(t,a) G_1(u,a) da = \\
&-4 \sum_{m,n=0}^\infty \binom{2m+1-t}{2m+1} \binom{2n-u}{2n} \frac{\zeta (2-t+2m) \zeta (1-u+2n)}{2^{s+2m+2n+1} (s+2m+2n+1)}
\end{split}
\end{equation*}
which coincides with the third infinite series in (\ref{eq:nmth4}). By changing variables, we obtain the fourth infinite series in (\ref{eq:nmth4}). 
Finally, we have
\begin{equation*}
\begin{split}
&\int_0^{1/2} G_2(s,a) G_2(t,a) G_1(u,a) da =  \\
&8 \! \sum_{l,m,n=0}^\infty \! \binom{2l+1-s}{2l+1} \binom{2m+1-t}{2m+1} \binom{2n-u}{2n} 
\frac{\zeta (2-s+2l) \zeta (2-t+2m) \zeta (1-u+2n)}{2^{2l+2m+2n+3} (2l+2m+2n+3)}
\end{split}
\end{equation*}
which provides the first infinite series in (\ref{eq:nmth3}). Hence, we obtain (\ref{eq:nmth3}). 
\end{proof}

\subsection{Proof of Corollaries}
\begin{proof}[Proof of Corollary \ref{cor:nm1}]
For simplicity, we put
$$
\alpha := e^{-\pi is}, \qquad \beta := e^{-\pi it}, \qquad \gamma := e^{-\pi iu}.
$$
Then, the equation (\ref{eq:mth1}) can be expressed as
\[
(1 + \alpha \beta \gamma) T(s,t,u) + (\alpha + \beta \gamma) T(u,s,t) + (\beta + \gamma \alpha) T(t,u,s) = S(s,t,u).
\]
Replacing variables $(s,t,u)$ by $(u,s,t)$ and $(t,u,s)$ in the formula above, we have
\begin{equation*}
\begin{split}
&(1 + \alpha \beta \gamma) T(u,s,t) + (\gamma + \alpha \beta) T(t,u,s) + (\alpha + \beta \gamma) T(s,t,u) = S(u,s,t), \\
&(1 + \alpha \beta \gamma) T(t,u,s) + (\beta + \gamma \alpha) T(s,t,u) + (\gamma + \alpha \beta) T(u,s,t) = S(t,u,s),
\end{split}
\end{equation*}
respectively. Therefore, it holds that
\[
\begin{pmatrix}
1 + \alpha \beta \gamma & \alpha + \beta \gamma & \beta + \gamma \alpha \\
\alpha + \beta \gamma & 1 + \alpha \beta \gamma & \gamma + \alpha \beta \\
\beta + \gamma \alpha & \gamma + \alpha \beta & 1 + \alpha \beta \gamma
\end{pmatrix} \!
\begin{pmatrix}
T(s,t,u) \\ T(u,s,t) \\ T(t,u,s)
\end{pmatrix}
=
\begin{pmatrix}
S_1(s,t,u) \\ S_1(u,s,t) \\ S_1(t,u,s)
\end{pmatrix}.
\]
Thus, we obtain (i) in Corollary \ref{cor:nm1} by Cramer's rule. Similarly, we can show that
\[
\begin{pmatrix}
\gamma + \alpha \beta & \beta + \gamma \alpha & \alpha + \beta \gamma  \\
\alpha + \beta \gamma & 1 + \alpha \beta \gamma & \gamma + \alpha \beta \\
\beta + \gamma \alpha & \gamma + \alpha \beta & 1 + \alpha \beta \gamma
\end{pmatrix} \!
\begin{pmatrix}
T(s,t,u) \\ T(u,s,t) \\ T(t,u,s)
\end{pmatrix}
=
\begin{pmatrix}
S_2(s,t,u) \\ S_1(u,s,t) \\ S_1(t,u,s)
\end{pmatrix}.
\]
from (\ref{eq:nmth1}) and (\ref{eq:nmth2}). Thus, we obtain the second equation in Corollary \ref{cor:nm1} by using Cramer's rule again. We can immediately show the third equation in Corollary \ref{cor:nm1}. The fourth equation in Corollary \ref{cor:nm1} is proved by Cramer's rule and
\[
\begin{pmatrix}
1 & 1 & 1 \\
\alpha + \beta \gamma & 1 + \alpha \beta \gamma & \gamma + \alpha \beta \\
\beta + \gamma \alpha & \gamma + \alpha \beta & 1 + \alpha \beta \gamma
\end{pmatrix} \!
\begin{pmatrix}
T(s,t,u) \\ T(u,s,t) \\ T(t,u,s)
\end{pmatrix}
=
\begin{pmatrix}
S_3(s,t,u) \\ S_1(u,s,t) \\ S_1(t,u,s)
\end{pmatrix}.
\]
The fifth equation in Corollary \ref{cor:nm1} is shown by Cramer's rule and
\[
\begin{pmatrix}
-1 & 1 & 1 \\
\alpha + \beta \gamma & 1 + \alpha \beta \gamma & \gamma + \alpha \beta \\
\beta + \gamma \alpha & \gamma + \alpha \beta & 1 + \alpha \beta \gamma
\end{pmatrix} \!
\begin{pmatrix}
T(s,t,u) \\ T(u,s,t) \\ T(t,u,s)
\end{pmatrix}
=
\begin{pmatrix}
S_4(s,t,u) \\ S_1(u,s,t) \\ S_1(t,u,s)
\end{pmatrix}.
\]
The sixth equality in Corollary \ref{cor:nm1} is proved by Cramer's rule and
\[
\begin{pmatrix}
1 + \alpha \beta \gamma & \alpha + \beta \gamma & \beta + \gamma \alpha \\
\gamma + \alpha \beta & \beta + \gamma \alpha & \alpha + \beta \gamma  \\
1 & 1 & 1 
\end{pmatrix} \!
\begin{pmatrix}
T(s,t,u) \\ T(u,s,t) \\ T(t,u,s)
\end{pmatrix}
=
\begin{pmatrix}
S_1(s,t,u) \\ S_2(s,t,u) \\ S_3(s,t,u)
\end{pmatrix}.
\]
The seventh equation in Corollary \ref{cor:nm1} is shown by Cramer's rule and
\[
\begin{pmatrix}
1 + \alpha \beta \gamma & \alpha + \beta \gamma & \beta + \gamma \alpha \\
\gamma + \alpha \beta & \beta + \gamma \alpha & \alpha + \beta \gamma  \\
-1 & 1 & 1 
\end{pmatrix} \!
\begin{pmatrix}
T(s,t,u) \\ T(u,s,t) \\ T(t,u,s)
\end{pmatrix}
=
\begin{pmatrix}
S_1(s,t,u) \\ S_2(s,t,u) \\ S_4(s,t,u)
\end{pmatrix}.
\]
Replacing variables $(s,t,u)$ by $(u,s,t)$ and $(t,u,s)$ in the definition of $S_4(s,t,u)$, we have
\begin{equation*}
\begin{split}
&S_4(u,s,t) = - T(u,s,t) + T(s,t,u) + T(t,u,s), \\
&S_4(t,u,s) = - T(t,u,s) + T(u,s,t) + T(s,t,u),
\end{split}
\end{equation*}
respectively. Hence, we obtain the eighth equation in Corollary \ref{cor:nm1}.
\end{proof}

\begin{proof}[Proof of Corollary \ref{cor:nm2}]
We consider the case $s=t=u$ in (\ref{eq:nmth1}), (\ref{eq:nmth2}), (\ref{eq:nmth3}) and (\ref{eq:nmth4}) to show Corollary \ref{cor:nm2}.
Clearly, we have
\[
\lim_{\varepsilon \to 0} S_1(\varepsilon, \varepsilon, \varepsilon) = 
6\lim_{\varepsilon \to 0} T(\varepsilon, \varepsilon, \varepsilon) .
\]
Taking $\varepsilon \to 0$ in the right-hand side of (\ref{eq:nmth1}) with $s=t=u= \varepsilon$, we have
\[
G(\varepsilon, \varepsilon, \varepsilon) \sum_{l,m,n=0}^\infty 
\frac{\kappa_{l,m,n} \eta_l^+(\varepsilon) \eta_m^+(\varepsilon) \eta_n^-(\varepsilon)}{l+m+n+1} \to -1 ,
\]
\[
G(\varepsilon, \varepsilon, \varepsilon) \sum_{n=0}^\infty \frac{2^{1-2\varepsilon} \eta_n^+(\varepsilon)}{\varepsilon+\varepsilon+n-1} \to 0, \qquad
G(\varepsilon, \varepsilon, \varepsilon) \sum_{m,n=0}^\infty 
\frac{\eta_m^-(\varepsilon) \eta_n^+ (\varepsilon)}{2^\varepsilon(\varepsilon+m+n)} \to 1
\]
from the fact that $\varepsilon \zeta (1 + \varepsilon) \to 1$ or $\varepsilon \zeta (1 - \varepsilon) \to -1$. Hence, we have the first formula of Corollary \ref{cor:nm2}. It should be emphasised that this formula can be also proved by (\ref{eq:nmth2}) and (\ref{eq:nmth3}) with $s=t=u= \varepsilon$ and $\varepsilon \to 0$.

Note that the function $\eta_k^\pm (s)$ is analytic for all $k \in {\mathbb{Z}}_{\ge 0}$ and $0 \ne s \in {\mathbb{C}}$ from the definition. At $s=n \in 2{\mathbb{Z}}_{<0}$, the functions $G(s,s,s)$ and $G_{\! ccc} (s,s,s)$ have a triple zero. However, the functions within the parenthesis of the right-hand side of (\ref{eq:nmth1}), (\ref{eq:nmth2}) and (\ref{eq:nmth3}) have poles, whose orders are at most $1$, there. Thus, we obtain the second formula of Corollary \ref{cor:nm2} when $n$ is a negative even integer. Next, suppose $n$ is a negative odd integer. Then, the function $G_{\! ssc} (s,s,s)$ have a double zero at $s=n$. By contrast, the functions within the brackets of the right-hand side of (\ref{eq:nmth4}) have poles, whose orders are at most $1$, there. Therefore, we obtain the second formula of Corollary \ref{cor:nm2}.
\end{proof}

\begin{proof}[Proof of Corollaries \ref{cor:pole1} and \ref{cor:nm3}]
Clearly, the function $T(s,s,s)$ have a simple pole at $s= 2/3$ by (\ref{eq:nmth3}) with $s=t=u$. We can see that $T(s,s,s)$ has simple poles at $s= 1/2 -1/2, -3/2, -5/2, \ldots$ because the function $G_{\! ccc} (s,s,s)$ has no zeros, the fifth, sixth and seventh infinite series within the parenthesis of the right-hand side of (\ref{eq:nmth3}) with $s=t=u$ have simple poles, and the first term and other infinite series within the brackets of the right-hand side of (\ref{eq:nmth3}) with $s=t=u$ have no poles, there. There are no other poles from the first statement of Theorem \ref{th:nm2}. Thus, we have Corollary \ref{cor:pole1}.

In addition, we can easily see that the function $T(s,s,s)$ can not be written by a polynomial in the form of (\ref{eq:poly1}) since any polynomial expressed as (\ref{eq:poly1}) does not have infinitely many poles. Hence, we obtain Corollary \ref{cor:nm3}.
\end{proof}

\subsection*{Acknowledgments}
The author is partially supported by JSPS grant 16K05077. 



\begin{thebibliography}{1}
\bibitem{Apo} 
T.~M.~Apostol, \textit{Introduction to Analytic Number Theory}. Undergraduate Texts in Mathematics, Springer, New York, 1976.
\bibitem{Bo}
J.~M.~Borwein, Hilbert's inequality and Witten's zeta-function. {\it{Am.~Math.~Monthly}} {\bf{115}} (2008), no.~2, 125--137.
\bibitem{BoDi}
J.~M.~Borwein and K.~Dilcher, Derivatives and fast evaluation of the Tornheim zeta function, {\it{Ramanujan J}}. {\bf{45}} (2018),  no.~2, 413--432. 
\bibitem{EM}
O.~Espinosa and V.~H.~Moll, {\it{The evaluation of Tornheim double sums. Part 1,}} Journal of Number Theory, {\bf{116}} (2006), 200--229.
\bibitem{FKMT1} 
H.~Furusho, Y.~Komori, K.~Matsumoto and H.~Tsumura, Desingularization of complex multiple zeta-functions. {\it{Amer.~J.~Math}}.  {\bf{139}} (2017),  no.~1, 147--173. 
\bibitem{FKMT2} 
H.~Furusho, Y.~Komori, K.~Matsumoto and H.~Tsumura, Desingularization of multiple zeta-functions of generalized Hurwitz-Lerch type and evaluation of $p$-adic multiple $L$-functions at arbitrary integers. {\it{Various aspects of multiple zeta values}},  27--66, RIMS Kokyuroku Bessatsu, {\bf{B68}}, Res.~Inst.~Math.~Sci. (RIMS), Kyoto, 2017.
\bibitem{LauGa}
A.~Laurin\v{c}ikas and R.~Garunk\v{s}tis, {\it{The Lerch zeta-function}}. Kluwer Academic Publishers, Dordrecht, 2002.
\bibitem{MaM02}
K.~Matsumoto, {\it{On the analytic continuation of various multiple zeta-functions}} in: Number Theory for the Millennium II, Proc. of the Millennial Conference on Number Theory, M. A. Bennett et. al. (eds.), A. K. Peters, (2002), 417-440. 
\bibitem{MNOT}
K.~Matsumoto, and T.~Nakamura, H.~Ochiai and H.~Tsumura, On value-relations, functional relations and singularities of Mordell-Tornheim and related triple zeta-functions. {\it{Acta Arith}}. {\bf{132}} (2008), no.~2, 99--125. 
\bibitem{Na}
T.~Nakamura, A functional relation for the Tornheim double zeta function. {\it{Acta Arithmetica}}. {\bf{125}} (2006), no. 3, 257--263. 
\bibitem{NH}
T.~Nakamura, Symmetric Tornheim double zeta functions, to appear in {\it{Abhandlungen aus dem Mathematischen Seminar der Universit\"{a}t Hamburg}}.
\bibitem{OR}
K.~Onodera, Mordell-Tornheim multiple zeta values at non-positive integers. {\it{Ramanujan J.}} {\bf{32}} (2013), no.~2, 221--226.
\bibitem{Onodera}
K.~Onodera, On generalized Mordell-Tornheim zeta functions. {\it{Ramanujan J}}. {\bf{47}} (2018), no.~1, 201--219. 
\bibitem{Romik}
D.~Romik, On the number of n-dimensional representations of $SU(3)$, the Bernoulli numbers, and the Witten zeta function, {\it{Acta Arithmetica}}. {\bf{180}} (2017), 111--159
\bibitem{SriCho}
H.~M.~Srivastava and J.~Choi, {\it{Zeta and q-Zeta functions and associated series and integrals}}. Elsevier, Inc., Amsterdam, 2012. 
\bibitem{Wiki}
Wikipedia, {\it{Riemann zeta function}}. \url{https://en.wikipedia.org/wiki/Riemann_zeta_function}
\end{thebibliography}
\end{document}